\documentclass{amsart}
\usepackage{amsmath, amsthm, amscd, amssymb,latexsym,enumerate}
\usepackage{url}
\usepackage{array}
\usepackage[all]{xy}
\usepackage{enumitem}
\usepackage{hyperref}

\def\Z{\mathbb{Z}}

\def\SS{{\mathcal S}}

\def\Z{{\mathbb Z}}

\def\Ann{\mathrm{Ann}}

\def\onto{\twoheadrightarrow}

\def\isom{\xrightarrow{\sim}}

\numberwithin{equation}{section}

\newtheorem{thm}[equation]{Theorem}
\newtheorem{lem}[equation]{Lemma}

\newtheorem{prop}[equation]{Proposition}

\theoremstyle{definition}

\newtheorem{algorithm}[equation]{Algorithm}

\title[Determining cyclicity of finite modules]{Determining cyclicity of finite modules}
\author[H.\ W.\ Lenstra, Jr.]{H.\ W.\ Lenstra, Jr.}
\address{Mathematisch Instituut,
Universiteit Leiden,
Postbus 9512,
2300 RA Leiden,
The Netherlands}
\email{hwl@math.leidenuniv.nl}
\author[A.\ Silverberg]{A.\ Silverberg}
\address{Department of Mathematics, Rowland Hall,
University of California, %Irvine,
Irvine, CA 92697, USA}
\email{asilverb@uci.edu}

\begin{document}

\thanks{This material is based on research sponsored by DARPA under agreement numbers FA8750-11-1-0248 and FA8750-13-2-0054 and by the Alfred P.~Sloan Foundation. The U.S.\ Government is authorized to reproduce and distribute reprints for Governmental purposes notwithstanding any copyright notation thereon. The views and conclusions contained herein are those of the authors and should not be interpreted as necessarily representing the official policies or endorsements, either expressed or implied, of DARPA or the U.S.\ Government.
}

\begin{abstract} 
We present a deterministic polynomial-time algorithm that determines whether a finite module 
over a finite commutative ring is cyclic, and if it is, outputs a generator.
\end{abstract}

\maketitle

\section{Introduction}

If $R$ is a commutative ring, then an $R$-module $M$ is
cyclic if there exists $y\in M$ such that $M=Ry$.

\begin{thm}
\label{introthm}
There is a deterministic polynomial-time algorithm that, given
a finite commutative ring $R$ and a finite $R$-module $M$,
decides whether there exists $y\in M$ such that $M = Ry$, and if there is, 
finds such a $y$.
\end{thm}

We present the algorithm in Algorithm \ref{finiteringsalgor} below.
The inputs are given as follows.
The ring $R$ is given as an abelian group 
by generators and relations,
along with all the products of pairs of generators.
The finite $R$-module $M$ is given as an abelian group, 
and for all generators of the abelian groups $R$ and 
all generators of the abelian group $M$
we are given the module products in $M$.

Our algorithm depends on $R$ being an Artin ring, and should
generalize to finitely generated modules over any commutative Artin ring
that is computationally accessible.

Theorem \ref{introthm} is one of the ingredients of our work
\cite{LenSil,LwS} on lattices with symmetry, and a sketch of the proof
is contained in \cite{LenSil}. 
Previously published algorithms of the same nature appear to
restrict to rings that are algebras over fields.
Subsequently to \cite{LenSil}, I.~Cioc\u{a}nea-Teodorescu \cite{Iuliana},
using different and more elaborate techniques,
greatly generalized our result, dropping the commutativity
assumption on the finite ring $R$ and finding, for any given
finite $R$-module $M$, a set of generators for $M$ of smallest possible size. 

See Chapter 8 of 
\cite{AtiyahMcD} for commutative algebra background.
For the purposes of this paper, commutative rings have an identity
element $1$, which may be $0$. 

\section{Lemmas on commutative rings}

If $R$ is a commutative ring and ${\boldsymbol{a}}$ is an ideal in $R$, let
$\Ann_R{\boldsymbol{a}}$ denote the annihilator of ${\boldsymbol{a}}$ in $R$.
We will use that every finite commutative ring is an Artin ring, 
that every Artin ring is isomorphic to a finite direct product of
local Artin rings, and that the maximal ideal in a local Artin ring
is always nilpotent. 

\begin{lem}
\label{localsquarelem}
If $A$ is a local Artin ring,
${\boldsymbol{a}}$ is an ideal in $A$, and  
${\boldsymbol{a}}^2 = {\boldsymbol{a}}$,
then ${\boldsymbol{a}}$ is $0$ or $A$.
\end{lem}

\begin{proof}
If ${\boldsymbol{a}}$ contains a unit, then ${\boldsymbol{a}}=A$. 
Otherwise, ${\boldsymbol{a}}$ is contained in the maximal ideal ${\boldsymbol{m}}$,
which is nilpotent. Thus there is an $r\in \Z_{>0}$ such that
${\boldsymbol{m}}^r=0$. 
Now ${\boldsymbol{a}} = {\boldsymbol{a}}^2 = \cdots = {\boldsymbol{a}}^r \subset
{\boldsymbol{m}}^r = 0$.
\end{proof}

\begin{lem}
\label{nonzeronilpotlem}
Suppose that $A$ is a finite commutative ring,
${\boldsymbol{a}}$ is an ideal in $A$, 
${\boldsymbol{b}} = \Ann_A{\boldsymbol{a}}$,
and  
${\boldsymbol{a}} \cap {\boldsymbol{b}} = 0$.
Then:
\begin{enumerate}
\item 
${\boldsymbol{a}}^2 = {\boldsymbol{a}}$;
\item
there is an idempotent $e\in A$
such that ${\boldsymbol{a}}= eA$, ${\boldsymbol{b}}=(1-e)A$, and
$A = (1-e)A \oplus eA = {\boldsymbol{b}} \oplus {\boldsymbol{a}}$;
\item
if ${\boldsymbol{b}}=0$ then ${\boldsymbol{a}}=A$.
\end{enumerate}
\end{lem}

\begin{proof}
Write $A$ as a finite direct product of
local Artin rings $A_1 \times \cdots \times A_s$. 
Then ${\boldsymbol{a}}$ is a direct product 
${\boldsymbol{a}}_1 \times \cdots \times {\boldsymbol{a}}_s$
of ideals 
${\boldsymbol{a}}_i \subset A_i$. 
Assume ${\boldsymbol{a}}^2 \neq {\boldsymbol{a}}$.
Then there is an $i$ such that 
${\boldsymbol{a}}_i^2 \neq {\boldsymbol{a}}_i$.
Let ${\boldsymbol{b}}_i = \Ann_{A_i}{\boldsymbol{a}}_i$.
Since ${\boldsymbol{a}} \cap {\boldsymbol{b}} = 0$, it follows that 
${\boldsymbol{a}}_i \cap {\boldsymbol{b}}_i = 0$.
Since $A_i$ is a local ring,  ${\boldsymbol{a}}_i$ is 
contained in the maximal ideal of $A_i$,
so ${\boldsymbol{a}}_i$ is nilpotent. 
Let $r$ denote the smallest positive integer such that 
${\boldsymbol{a}}_i^r = 0$. Since ${\boldsymbol{a}}_i \neq 0$ we have
$r \ge 2$.
Then ${\boldsymbol{a}}_i^{r-1}$ is contained in ${\boldsymbol{a}}_i$
and kills ${\boldsymbol{a}}_i$, so 
$0 \neq {\boldsymbol{a}}_i^{r-1} \subset {\boldsymbol{a}}_i \cap {\boldsymbol{b}}_i = 0$,
a contradiction. This gives (i).

Since $A$ is a finite product of local Artin rings, 
${\boldsymbol{a}}$ is generated by an idempotent $e$,
by Lemma \ref{localsquarelem}.
Then ${\boldsymbol{b}}=(1-e)A$ and
$A = (1-e)A \oplus eA = {\boldsymbol{b}} \oplus {\boldsymbol{a}}$.
This gives (ii) and (iii).
\end{proof}

\section{Preparatory lemmas}

If $R$ is a commutative ring, then a commutative $R$-algebra is a commutative ring
$A$ equipped with a ring homomorphism from $R$ to $A$.
Whenever $A$ is an $R$-algebra, we let $M_A$ denote
the $A$-module $A \otimes_R M$.

From now on,
suppose $R$ is finite commutative ring and $M$ is a finite $R$-module.
Let $\SS$ denote the set of quadruples $(A,B,y,N)$
such that:
\begin{enumerate}
\item
 $A$ and $B$ are finite commutative $R$-algebras for which
the natural map 
$f : R \onto A \times B$ is surjective and has 
nilpotent kernel, 
\item
$y\in M$ is such that the map
$B \to M_B = B \otimes_R M$ defined by $b \mapsto b\otimes y$ is an isomorphism
and such that $1\otimes y= 0$ in $M_A$,
\item
and $N$ is a submodule of $M$ such that
the natural map $N \to M_A$ defined by $z \mapsto 1\otimes z$ is onto
and such that the natural map $N \to M_B$ is the zero map.
\end{enumerate}
In Algorithm \ref{finiteringsalgor} below,
initially we take $(A,B,y,N) = (R,0,0,M)$. 
Clearly, $(R,0,0,M)\in\SS$. 
Throughout that algorithm, we always have $(A,B,y,N)\in\SS$.
While $A$ and $B$ occur in the proof of correctness of Algorithm \ref{finiteringsalgor},
the $R$-algebra $B$ does not actually occur in the algorithm itself.

\begin{lem}
\label{SMA0lem}
If $(A,B,y,N)\in \SS$ and $M_A=0$, then $M=Ry$.
\end{lem}

\begin{proof}
Let $J$ denote the kernel of $f : R \onto A \times B$, and 
let
$I_A$ (resp., $I_B$) denote the kernel of the composition
of $f$ with projection from $A \times B$ onto $A$ (resp., $B$).
Since  $J$ is nilpotent we have $J^r = 0$ for some $r\in \Z_{>0}$.
Since  
$0 = M_A = A \otimes_R M = (R/I_A) \otimes_R M \cong M/I_AM$
it follows that $I_AM = M$ 
Since $JM \subseteq I_BM = I_BI_AM \subseteq (I_B \cap I_A)M = JM$, it follows that
$JM = I_BM$.
Letting  $y' = (y \mod I_BM) \in M/I_BM$, then $M_B \cong M/I_BM = By'$. Thus,
\begin{multline*}
M = Ry + I_BM = Ry + JM = Ry + J(Ry + JM) \\ = Ry + J^2M = \ldots
 = Ry + J^rM = Ry.
\end{multline*}
\end{proof}

\begin{lem}
\label{xexistslem}
Suppose $(A,B,y,N)\in \SS$ and $M_A \neq 0$. Then there exists 
$x\in N$ such that $1 \otimes x \neq 0$ in
$M_A$. Choosing $x$ and
letting ${\boldsymbol{a}} = \Ann_A (1 \otimes x)$ and 
${\boldsymbol{b}} = \Ann_A{\boldsymbol{a}}$, we have:
\begin{enumerate}
\item
$(A/({\boldsymbol{a}} \cap {\boldsymbol{b}}),B,y,N)\in\SS;$
\item
If ${\boldsymbol{a}} \cap {\boldsymbol{b}} =0$ and
$(A/{\boldsymbol{a}}) \otimes x = M_{A/{\boldsymbol{a}}}$, 
then
$(A/{\boldsymbol{b}},(A/{\boldsymbol{a}}) \times B,x+y,{\boldsymbol{a}}N)\in\SS$,
where ${\boldsymbol{a}}N$ denotes $f^{-1}({\boldsymbol{a}}\times B)N$.
\item
If ${\boldsymbol{a}} \cap {\boldsymbol{b}} =0$ and
$(A/{\boldsymbol{a}}) \otimes x \neq M_{A/{\boldsymbol{a}}}$, 
then $M$ is not cyclic.
\end{enumerate}
\end{lem}

\begin{proof}
Since the map $N \to M_A$, $z \mapsto 1\otimes z$ is onto, 
as long as $M_A \neq 0$ there exists $x\in N$ such that $1 \otimes x \neq 0$ in
$M_A$. 

Since ${\boldsymbol{a}}{\boldsymbol{b}} = 0$, we have
$({\boldsymbol{a}} \cap {\boldsymbol{b}})^2 = 0$, so
${\boldsymbol{a}} \cap {\boldsymbol{b}}$ is a nilpotent ideal in $A$.
It follows that $(A/({\boldsymbol{a}} \cap {\boldsymbol{b}}),B,y,N)\in\SS$,
giving (i).

From now on, suppose that ${\boldsymbol{a}} \cap {\boldsymbol{b}} = 0$.
By Lemma \ref{nonzeronilpotlem},
there is an idempotent $e\in A$
such that ${\boldsymbol{a}}= eA$, ${\boldsymbol{b}}=(1-e)A$, and
$A = (1-e)A \oplus eA = {\boldsymbol{b}} \oplus {\boldsymbol{a}}$.
It follows that
$A \isom A/{\boldsymbol{a}} \times A/{\boldsymbol{b}}$, so
$M_A \isom M_{A/{\boldsymbol{a}}} \times M_{A/{\boldsymbol{b}}}$.
If $(x',x'')$ is the image of $1 \otimes x$ under the latter map, then  $x''=0$
(we have ${\boldsymbol{b}}x''=0$ since $x''\in (A/{\boldsymbol{b}})\otimes_R M$,
and ${\boldsymbol{a}}x''=0$ since ${\boldsymbol{a}}(1 \otimes x)=0$;
thus $Ax'' = ({\boldsymbol{a}}+{\boldsymbol{b}})x''=0$, so $x''=0$).
The map $i_{\boldsymbol{a}} : A/{\boldsymbol{a}}\to M_{A/{\boldsymbol{a}}}$
defined by $i_{\boldsymbol{a}}(t) = t x' = t\otimes x$
is injective since $\Ann_{A/{\boldsymbol{a}}} x' = 0$.

First suppose $(A/{\boldsymbol{a}}) \otimes x = M_{A/{\boldsymbol{a}}}$.
Then the injective map $i_{\boldsymbol{a}}$ is an isomorphism.
Since $0 = x'' = 1_{A/{\boldsymbol{b}}} \otimes x$, we have
$1 \otimes (x+y) = 0$ in $M_{A/{\boldsymbol{b}}}$.
It is now easy to check that
$(A/{\boldsymbol{b}},(A/{\boldsymbol{a}}) \times B,x+y,{\boldsymbol{a}}N)\in\SS$,
giving (ii).
Note that ${\boldsymbol{b}} \neq 0$ (if ${\boldsymbol{b}} = 0$,
then ${\boldsymbol{a}} = A$ by Lemma \ref{nonzeronilpotlem}, 
contradicting that $1\otimes x\neq 0$ in $M_A$).

Now suppose that $(A/{\boldsymbol{a}}) \otimes x \neq M_{A/{\boldsymbol{a}}}$.
By way of contradiction, suppose $M$ is a cyclic $R$-module.
Then $M_{A/{\boldsymbol{a}}}$ is a cyclic $A/{\boldsymbol{a}}$-module.
Since the domain and codomain of 
$i_{\boldsymbol{a}} : A/{\boldsymbol{a}}\hookrightarrow M_{A/{\boldsymbol{a}}}$ 
are both finite, 
it now follows that
$i_{\boldsymbol{a}}$ is surjective, so
$(A/{\boldsymbol{a}}) \otimes x = M_{A/{\boldsymbol{a}}}$.
This contradiction gives (iii).
\end{proof}

The intuition behind Algorithm \ref{finiteringsalgor}
is that throughout the algorithm, 
$y$ generates the ``non-$A$ part'' of $M$, and
the goal is to shrink the ``$A$-part'' of $M$, namely $N$.

\section{Main algorithm}

\begin{algorithm}
\label{finiteringsalgor}
Input a finite commutative ring $R$ 
and a finite $R$-module $M$.
Decide whether there exists $y\in M$ such that $M = Ry$, and if there is, 
find such a $y$.

\begin{enumerate}
\item
Initially, take $A=R$, $y=0$, and $N=M$.
\item
If $M_A = 
0$, stop and output ``yes'' with generator $y$.
\item
Otherwise, pick $x\in N$ such that $1 \otimes x \neq 0$ in
$M_A$,
and compute
${\boldsymbol{a}} = \Ann_A (1 \otimes x)$, 
${\boldsymbol{b}} = \Ann_A{\boldsymbol{a}}$,
and ${\boldsymbol{a}} \cap {\boldsymbol{b}}$.
\item
If ${\boldsymbol{a}} \cap {\boldsymbol{b}} \neq 0$, 
replace 
$A$ by $A/({\boldsymbol{a}} \cap {\boldsymbol{b}})$ and 
go back to step (ii).
\item
If ${\boldsymbol{a}} \cap {\boldsymbol{b}} = 0$, 
then if $(A/{\boldsymbol{a}}) \otimes x \neq M_{A/{\boldsymbol{a}}}$
terminate with ``no'', and
if $(A/{\boldsymbol{a}}) \otimes x = M_{A/{\boldsymbol{a}}}$
replace $A$, $y$, and $N$ by 
$A/{\boldsymbol{b}}$, $x + y$, and ${\boldsymbol{a}}N$, respectively, and 
go back to step (ii).
\end{enumerate}
\end{algorithm}

\begin{prop}
\label{finiteringsprop}
Algorithm \ref{finiteringsalgor} runs in polynomial time,
and on input a finite commutative ring $R$ 
and a finite $R$-module $M$,
decides whether there exists $y\in M$ such that $M = Ry$, and if there is, 
finds such a $y$.
\end{prop}

\begin{proof}
Since $A$ is a finite ring, if the algorithm does not  stop with ``no'' 
then eventually $A=0$ and $M_A=0$.
Step (ii) of the algorithm is justified
by Lemma \ref{SMA0lem}, while steps (iii), (iv), and (v) are justified
by Lemma \ref{xexistslem}.

The computations of annihilators and of the decompositions
$A \isom A/{\boldsymbol{a}} \times A/{\boldsymbol{b}}$
can be done in polynomial time using linear algebra (see \S 14 of \cite{HWLMSRI});
in particular, ${\boldsymbol{a}}$ is the kernel of the map
$A \to M_A$ defined by $t \mapsto t(1\otimes x)$.
For any $B$, compute $M_B$ by computing $M/I_BM$ (and analogously
for $M_A$).
Each new $A$ is at most half the size of the  $A$ it replaces. 
This implies that the number of steps is at most linear in the length
of the input.
\end{proof}

\end{document}